\documentclass{amsart}

\usepackage{amssymb,amsfonts, latexsym,amsmath,hhline,array,longtable}
\usepackage{pdfsync,color, comment,colortbl, mathrsfs,stmaryrd,cite,graphicx}

\usepackage{ifpdf}
\ifpdf \usepackage[colorlinks=true, citecolor=blue, linkcolor=blue, urlcolor=blue]{hyperref} \fi

\newcommand{\cal}{\mathcal}

\newtheorem{formula}{}[section]
\newtheorem{definition}[formula]{Definition}
\newtheorem{corollary}[formula]{Corollary}
\newtheorem{remark}[formula]{Remark}
\newtheorem{lemma}[formula]{Lemma}
\newtheorem{theorem}[formula]{Theorem}
\newtheorem{prop}[formula]{Proposition}

\def\thrm{\begin{theorem}}
\def\thrml#1{\begin{theorem}\label{#1}}
\def\ethrm{\end{theorem}}
\def\rmrk{\begin{remark}}
\def\rmrkl#1{\begin{remark}\label{#1}}
\def\ermrk{\end{remark}}
\def\dfntn{\begin{definition}}
\def\dfntnl#1{\begin{definition}\label{#1}}
\def\edfntn{\end{definition}}
\def\nmrt{\begin{enumerate}}
\def\enmrt{\end{enumerate}}
\def\tm#1{\item[{\rm (#1)}]}
\def\qtnl#1{\begin{equation}\label{#1}}
\def\eqtn{\end{equation}}
\def\lmm{\begin{lemma}}
\def\lmml#1{\begin{lemma}\label{#1}}
\def\elmm{\end{lemma}}
\def\crllr{\begin{corollary}}
\def\crllrl#1{\begin{corollary}\label{#1}}
\def\ecrllr{\end{corollary}}
\def\css{\begin{cases}}
\def\ecss{\end{cases}}
\def\prf{\begin{proof}}
\def\eprf{\end{proof}}

\def\cC{{\cal C}}

\def\cX{{\cal X}}
\def\cY{{\cal Y}}

\DeclareMathOperator{\cyl}{cyl}
\DeclareMathOperator{\diag}{Diag}

\DeclareMathOperator{\sym}{Sym}

\DeclareMathOperator{\wl}{\mathsf{WL}}
\DeclareMathOperator{\WL}{\mathsf{WL}} 

\def\qaq{\quad\text{and}\quad}

\def\ov{\overline}
\def\sbs#1#2{_{\scriptscriptstyle{#1,#2}}}
\def\sbo#1{_{\scriptscriptstyle{#1}}}

\def\wh{\widehat}

\begin{document}

\title{On the Weisfeiler-Leman dimension of permutation graphs}
\author{}
\address{}
\author{Jin Guo}
\address{Hainan University, Haikou, China}
\email{guojinecho@163.com}
\author{Alexander L. Gavrilyuk}
\address{Shimane University, Matsue, Japan}
\email{gavrilyuk@riko.shimane-u.ac.jp}
\author{Ilia Ponomarenko}
\address{Hainan University, Haikou, China}
\address{Steklov Institute of Mathematics at St. Petersburg, Russia}
\email{inp@pdmi.ras.ru}
\thanks{The research of Alexander Gavrilyuk is supported by JSPS KAKENHI Grant Number 22K03403. 
The research of Jin Guo is supported by the National Natural Science Foundation of China 
(Grant No. 11961017).}
\date{}

\begin{abstract}
It is proved that the Weisfeiler-Leman dimension of the class of permutation graphs is at most 18. Previously it was only known that this dimension is finite (Gru{\ss}ien, 2017).
\end{abstract}

\maketitle

\section{Introduction}
The Weisfeiler-Leman (WL) algorithm, originally introduced 
in \cite{WL1968}, is a graph isomorphism test. 
It first colors the ordered pairs of vertices of 
the input graphs and then iteratively refines the coloring. 
The procedure is isomorphism-invariant, i.e., 
isomorphic graphs eventually receive the same canonical coloring. 
The converse is not true in general. 
Stronger isomorphism-invariant colorings can be obtained if 
one considers coloring the $d$-tuples of vertices with $d>2$; 
the corresponding generalization, called the $d${\bf -dimensional} 
Weisfeiler-Lehman algorithm (or the $d$-{\bf dim WL} for short), 
was introduced in the late 1970s (see \cite{CFI1992,B2016}) 
and played a crucial role 
in Babai’s breakthrough quasipolynomial-time 
isomorphism test \cite{B2015}. For $d = 1$ and $d = 2$, 
the $d$-dim WL is in fact the naive refinement and ordinary WL, respectively. 
The family of $d$-dim WL algorithms has striking 
connections with other areas in theoretical computer science 
and algebraic combinatorics (see \cite{EP2000,AM2013,GO2012, Conf2018}). 

Given a graph $X$, there exists the smallest integer $d$ 
such that the $d$-dim WL correctly identifies $X$. 
Following Grohe \cite{Grohe2017}, this $d$ is called 
the {\bf WL dimension} $\dim_{\mathsf{WL}}(X)$ of the graph $X$. 
The seminal paper of Cai, F\"urer and Immerman \cite{CFI1992} 
shows that the WL dimension in the class of all graphs cannot 
be bounded from above by a constant. However, graphs belonging 
to various interesting classes often have a bounded dimension. 

More precisely,
the WL dimension 
of a class $\mathcal{K}$ of graphs
is defined to be the maximum of $\dim_{\mathsf{WL}}(X)$ over all $X\in \mathcal{K}$ 
if it exists, or $\infty$ otherwise.
If $\dim_{\mathsf{WL}}(X)\leq m$, then the graph isomorphism problem 
restricted to $\mathcal{K}$ is solved by the $m$-dim WL, 
which, for fixed $m$, runs in polynomial time \cite{IL1990}. 
This yields a uniform purely combinatorial isomorphism test 
for all graph classes of bounded WL dimension.

By a celebrated result of Grohe \cite{Gr2010}, all graph classes with an excluded minor 
have a bounded WL dimension. In many recent studies, the explicit upper bounds on the WL 
dimension have been obtained for various graph classes: the interval graphs \cite{EPT2000}, 
planar graphs \cite{KPS2017}, graphs of given Euler genus \cite{GK2019}, or rank width \cite{GN2019}, 
to name a few. 
The graphs of WL dimension 1 have completely been characterized in \cite{AKRV2017} 
and independently in \cite{FKV2019} 
(however, such a characterization for the graphs of higher WL dimension seems unachievable~\cite{FKV2019}).

In the present paper, we study the WL dimension of the class of permutation graphs. 
This subclass of the perfect graphs has a large number of applications \cite{G} and 
has been intensely studied since the late 1960s \cite{PLE1971}.  
The original definition is somewhat cumbersome: 
a permutation graph is a graph $X$ such that there exists a labeling $\alpha_1,\ldots,\alpha_n$ of its vertices and a permutation $\pi\in \sym(n)$ 
satisfying $(i - j)(\pi(i) - \pi(j))<0$ if and only if $\alpha_i$ is adjacent to $\alpha_j$. 
For our purpose, we will make use of the following equivalent characterization (see \cite{PLE1971}): 
a graph $X$ is a permutation graph if and only if both $X$ and its complement 
are transitively orientable (see Section \ref{ssect:orienations} for further details).
Based on this property, Coulborn \cite{C1981} described a first polynomial-time algorithm 
for testing isomorphism of permutation graphs. 
Some other characterizations of the permutation graphs are known, e.g., 
as the comparability graphs of partially ordered sets of order dimension at most two \cite{BFR1971}: 
this was used by Spinrad \cite{S1985} 
to construct efficient recognition and 
isomorphism algorithms. 

A characterization of permutation graphs in terms of excluded induced subgraphs was obtained by Gallai \cite{G1967}. In the spirit of the above mentioned result by Grohe \cite{Gr2010}, 
Gru{\ss}ien \cite{Gruss2017} considered the question of whether there is a logic 
that captures polynomial time on all classes of graphs with excluded induced subgraphs. 
The main result obtained there implies that 
the Weisfeiler-Leman dimension of any permutation graph is bounded from above by an absolute constant; however, its value has not been estimated 
(and it is unclear whether it can be obtained within the approach of \cite{Gruss2017}).
Our main result is as follows.

\begin{theorem}\label{thm:perm}
The Weisfeiler-Leman dimension of the class of permutation graphs is at most $18$.
\end{theorem}

It is shown in \cite{PR2021} that the Weisfeiler-Leman dimension of 
bipartite permutation graphs is $2$ or $3$; thus, we expect that the bound 
from Theorem~\ref{thm:perm} can be improved.

The proof of Theorem~\ref{thm:perm} is based on the theory of coherent configurations and is presented in Section~\ref{sect:main}. Let us discuss the scheme of the proof briefly. Let $X$ be a (permutation) graph with vertex set~$\Omega$. The coherent configuration $\cX$ of~$X$ 
is the output of the 2-dim WL algorithm applied to $X$ and can be thought 
of as an arc-colored directed complete graph on $\Omega$. 

We were not able to prove that a permutation graph can be identified by the properties of its coherent configuration. To resolve this problem, 
we invoke the theory of extensions of coherent configurations 
(see \cite{CP} and Section~\ref{011122a} in the present paper 
for further details). For each positive integer~$m$, 
the coherent configuration~$\cX$ can canonically be extended to 
a coherent configuration $\wh{\cX}^{(m)}$ on the Cartesian 
power~$\Omega^m$. On the other hand, the number $\dim_{\mathsf{WL}}(X)$ is known to be 
less than or equal to the tripled so-called separability number 
of $\cX$, which is estimated from above by means of ~$\wh\cX^{(m)}$. 
For our purpose, we take $m=2$ and then it suffices to verify 
that the separability number of~$\cX$ is at most~$6$.

The proof goes by induction the base of which is the case when the permutation graph~$X$ is uniquely orientable. In Section~\ref{sect:uniq}, this property is  used to prove that the individualization of two elements belonging to~$\Omega^2$ is sufficient for the Weisfeiler-Leman algorithm to split the coherent configuration $\wh{\cX}^{(2)}$ into the discrete one. 
This in turn implies that the separability number of~$\cX$ is at most~$6$.

Now, let $X$ be a permutation graph which is not uniquely orientable. Then $X$ admits a non-trivial modular decomposition, and, as was implicitly proved in~\cite{C1981}, this decomposition  can be chosen in a canonical way. In Section~\ref{011122c}, we use this fact to prove that this decomposition ``can be seen'' inside the coherent configuration~$\wh{\cX}^{(2)}$. This enables us to guarantee that the separability number of~$\cX$ is less than or equal to the maximum separability number of the coherent configurations of the smaller graphs associated with the modular decomposition of $X$. 

The paper is organized as follows.
In order to make the proof as self-contained as possible, we present in Section~\ref{011122w} 
the necessary notation and facts about permutation graphs and in Section~
\ref{sect:cc&extensions} those from the theory of coherent configurations. 
More details about these topics can be found in the monographs~\cite{G} and~\cite{CP}, respectively. 
In Section~\ref{011122c} we study an interplay between modular decompositions 
of graphs and their coherent configurations, as applied to algebraic isomorphisms.
Sections~\ref{sect:uniq} and~\ref{sect:main} are devoted to the proof of Theorem \ref{thm:perm} 
by induction as described above.

\subsection*{Notation.}
Throughout the paper, $\Omega$ stands for a finite set. For $\Delta\subseteq \Omega$, the diagonal of the Cartesian product $\Delta\times\Delta$ is denoted by~$1_\Delta$. 
For an equivalence relation $e$ on a subset of $\Omega$, we denote by $\Omega/e$ 
the set of all classes of~$e$.

For a binary relation $r\subseteq\Omega\times\Omega$, 
we set 
$r^*=\{(\beta,\alpha)\!:\ (\alpha,\beta)\in r\}$, and $r^f=\{(\alpha^f,\beta^f)\!:\ (\alpha,\beta)\in r\}$ for any bijection $f$ from $\Omega$ to another set. 
The {\bf left support} of $r$ is defined to be the set 
$\{\alpha\in\Omega\!:\ \alpha r\ne\varnothing\}$,
where $\alpha r=\{\beta\in\Omega\!:\ (\alpha,\beta)\in r\}$ is the {\bf neighborhood} 
of a point $\alpha\in\Omega$ in the relation $r$. The {\bf right support} of 
$r$ is the left support of $r^*$.
For relations $r,s\subseteq\Omega\times\Omega$, we put 
$$r\cdot s=\{(\alpha,\beta)\!:\ (\alpha,\gamma)\in r,\ (\gamma,\beta)\in s\text{ for some }\gamma\in\Omega\}.$$ 
For $\Delta,\Gamma\subseteq \Omega$, we set $r_{\Delta,\Gamma}=r\cap (\Delta\times \Gamma)$ 
and abbreviate $r_{\Delta}:=r_{\Delta,\Delta}$.

For a set $S$ of relations on $\Omega$, 
we denote by $S^\cup$ the set of all unions of the elements of $S$ and 
put $S^*=\{s^*\!:\ s\in S\}$ and $S^f=\{s^f\!:  s\in S\}$ for any bijection $f$ from $\Omega$ 
to another set.

\section{Permutation graphs and their modular decomposition}\label{011122w}

\subsection{Basic definitions.}
By a {\bf graph} we mean a finite simple undirected graph, i.e., a pair $X=(\Omega,E)$
of a finite set $\Omega$ of vertices and an irreflexive symmetric relation $E\subseteq \Omega\times \Omega$. 
The elements of $E=:E(X)$ are called {\bf edges}\footnote{Traditionally, an \emph{edge} means 
a subset of two adjacent vertices. In this paper, by an edge we mean an ordered pair 
of adjacent vertices, which is usually called an \emph{arc}.}, 
and $E$ is the {\bf edge set} of the graph $X$. 
Two vertices $\alpha$ and $\beta$ are {\bf adjacent} (in $X$)
whenever $(\alpha,\beta)\in E$ (equivalently, $(\beta,\alpha)\in E$). 
The subgraph of $X$ induced by a set $\Delta\subseteq \Omega$ 
is denoted by $X_{\Delta}=(\Delta,E_{\Delta})$.
A graph $X=(\Omega,E)$ is {\bf connected} if the transitive reflexive closure of $E$ equals $\Omega^2$. 
The {\bf complement} of $X$ is the graph $\overline{X}=(\Omega,\overline{E})$ where 
$\overline{E}=\Omega^2\setminus \left(E\cup 1_{\Omega}\right)$.
A graph is {\bf coconnected} if its complement is connected. 

\subsection{Orientations of graphs.}\label{ssect:orienations}
An {\bf orientation} of a graph $X=(\Omega,E)$ is a subset $A\subseteq E$ such that $A\cap A^*=\varnothing$ and $A\cup A^*=E$. 
A orientation $A$ is called {\bf transitive} if $A\cdot A\subseteq A$, 
which means that $(\alpha,\beta), (\beta,\gamma)\in A$ implies $(\alpha,\gamma)\in A$.
A graph is called {\bf transitively orientable} (or a {\bf comparability graph}) 
if it admits a transitive orientation.
A transitive orientation of a complete graph is called 
a {\bf transitive tournament}.

Following \cite[Chapter~5]{G}, we define the $\Gamma$-{\bf relation} 
of $X$ as a binary relation on $E$ consisting of all pairs 
$(\alpha,\beta)$ and $(\alpha',\beta')$ such that
\[
\quad\left(\alpha=\alpha'\mbox{ and } (\beta,\beta')\notin E\right)
\quad
\mbox{or}\quad \left(\beta=\beta'\mbox{ and } (\alpha,\alpha')\notin E\right).
\]
It is easily seen that this relation is symmetric and reflexive.
The transitive closure of the $\Gamma$-relation is an equivalence 
relation and hence partitions $E$ into equivalence classes 
called the {\bf implication classes} of $X$. 
The implication class of $X$ containing an edge $\mathbf{e}\in E$ 
is denoted by $I_X(\mathbf{e})$. In what follows, given a graph $X$, 
we abbreviate, $I(\mathbf{e})=I_X(\mathbf{e})$ if $\mathbf{e}\in E$ 
and $I(\mathbf{e})=I_{\overline{X}}(\mathbf{e})$ if $\mathbf{e}\in \overline{E}$. 

\begin{lemma}\label{lm:implicationclasses}{\rm(\hspace{1sp}\cite[Theorem~5.1, Theorem~5.4]{G})}

Let $X$ be a transitively orientable graph and $\mathbf{e}$ any of its edges. Then:
\begin{itemize}
    \item[(i)] $I_X(\mathbf{e})\cap I_X(\mathbf{e})^*=\varnothing$;
    \item[(ii)] if $A$ is a transitive orientation of $X$, then 
$I_X(\mathbf{e})\subseteq A$ or $I_X(\mathbf{e})^*\subseteq A$. 
\end{itemize}
\end{lemma}

A graph is said to be {\bf uniquely partially orderable (UPO)} 
if it admits 
at most two transitive orientations. (Note that if a UPO graph contains 
at least one edge, then there are exactly two transitive orientations, 
one being the reversal of the other.) 
For example, a connected bipartite graph is always UPO, 
while a complete graph $K_n$ on $n$ vertices is UPO 
if and only if $n\leq 2$.

A graph $X$ is called a {\bf permutation graph} if and only if $X$ and $\overline{X}$ 
are transitively orientable. 
Clearly, the class of permutation graph is closed under taking 
complement and induced subgraphs.
A (permutation) graph $X$ is {\bf uniquely orientable} if both $X$ and $\overline{X}$ 
are UPO graphs. It follows from Lemma \ref{lm:implicationclasses} that a permutation graph 
$X=(\Omega,E)$ is uniquely orientable if and only if 
\begin{equation}\label{eq:uniquely}
I_X(\mathbf{e})\cup I_X(\mathbf{e})^*\cup I_{\overline{X}}(\mathbf{f})\cup I_{\overline{X}}(\mathbf{f})^*=\Omega^2\setminus 1_{\Omega}
\end{equation}
holds for every $\mathbf{e}\in E$ and every $\mathbf{f}\in \overline{E}$.


\begin{lemma}{\rm (cf. \cite[Chapter~5, Exercise~5]{G})}\label{lemma-tournament}
Let $A$ and $\overline{A}$ be transitive orientations of a (permutation) graph
and its complement, respectively. Then $A\cup \overline{A}$ is 
a transitive tournament.
\end{lemma}
\begin{proof} 
Clearly, the relation $A\cup \overline{A}$ is anti-symmetric and irreflexive. 
To prove that it is transitive, it suffices to show that it does not contain 
a cycle of length $3$. Suppose on the contrary that 
$(\alpha,\beta),(\beta,\gamma),(\gamma,\alpha)\in A\cup \overline{A}$. 
Without loss of generality, we may assume $(\alpha,\beta),(\beta,\gamma)\in A$. 
Then $(\alpha,\gamma)\in A$, as $A$ is a transitive orientation of the graph, 
a contradiction.
\end{proof}

\subsection{Composition graph.}
Let $X_0$ be a graph with vertex set $\{\alpha_1,\alpha_2,\ldots,\alpha_n\}$ 
and let $\cC=\{X_1,\ldots,X_n\}$ be a collection of graphs whose vertex sets 
are pairwise disjoint. The {\bf composition graph} 
\[
X:=X_0[X_1,X_2,\ldots,X_n]=X_0[\cC]
\] 
is formed as follows: for all $1 \leq i,j \leq n$, replace the vertex $\alpha_i$ 
in $X_0$ with the graph $X_i$ and join each vertex of $X_i$ to each 
vertex of $X_j$ whenever $\alpha_i$ is adjacent to $\alpha_j$ in $X_0$. 
The vertex set $\Omega_i$ of the graph $X_i$, $i=1,\ldots,n$, 
is {\bf partitive}, i.e., every vertex of $X$, not belonging to $\Omega_i$, 
is adjacent to either all or none of the vertices of $\Omega_i$.  
The sets $\Omega_i$ are called the {\bf modules} of $X$, and 
the corresponding decomposition of the vertex set of $X$ is called {\bf modular}. 
The composition graph $X$ is {\bf non-trivial} if $n\geq 2$ and 
$|\Omega_i|\geq 2$ for at least one $i$.

\begin{lemma}\label{lm:eE}
Let $e$ be an equivalence relation on $\Omega$. 
Then a graph $(\Omega,E)$ is a composition graph 
whose modules are the equivalence classes of $e$ 
if and only if 
\begin{equation}\label{eq:eE}
e\cdot (E-e)=E-e=(E-e)\cdot e.
\end{equation}
\end{lemma}
\begin{proof}
One has $(\alpha,\beta)\in E-e$ if and only if $(\alpha,\beta)\in E$ 
and $\alpha,\beta$ belong to different equivalence classes of $e$. 
Thus, Eq. \eqref{eq:eE} holds if and only if, 
for every two distinct $\Delta,\Gamma\in \Omega/e$, 
the graph $(\Delta\cup \Gamma, E_{\Delta,\Gamma})$ is empty or complete bipartite.
The latter condition holds if and only if $(\Omega,E)$ is a composition 
graph whose modules are the classes of $\Omega/e$.
\end{proof}

In what follows, a graph $X=(\Omega,E)$ satisfying Eq. \eqref{eq:eE} 
for some equivalence relation $e$ will be refereed to as 
a {\bf composition graph with respect to} $e$. 
Note that this is equivalent to saying that $X=X_0[\cC]$ 
where $\cC=\{X_{\Delta}\mid \Delta\in \Omega/e\}$ 
and $X_0$ is the {\bf quotient graph} of $X$ modulo $e$ 
(that is, the vertex set of $X_0$ is $\Omega/e$ 
and two distinct vertices are adjacent in $X_0$ 
whenever there is an edge between them).

Let $X=(\Omega,E)$ be a graph. 
Two vertices $\alpha,\beta$ are called $0$-{\bf twins} ($1$-{\bf twins}, respectively) 
if $\alpha,\beta$ are not adjacent (are adjacent, respectively) and, 
for any vertex $\gamma\in\Omega\setminus\{\alpha,\beta\}$,
the set $\gamma E$ contains either both $\alpha$ and $\beta$ or none of them.
A graph is {\bf irreducible} if it contains neither distinct $0$-twins nor $1$-twins, 
and {\bf reducible} otherwise.
The relation ``{\it to be $0$-twins in $X$}'' is an equivalence relation on $\Omega$, 
called the {\bf $0$-equivalence} of~$X$.
The relation ``{\it to coincide or to be $1$-twins in $X$}'' is also 
an equivalence relation on $\Omega$, called the {\bf $1$-equivalence} of~$X$.
The following statement is obvious.

\begin{lemma}\label{lm:reducible}
A reducible graph is a non-trivial composition graph whose modules are {either all classes of the $0$-equivalence or all classes of the $1$-equivalence}. 
\end{lemma}

Let $P_m$ denote a path graph on $m$ vertices.

\begin{lemma}\label{lm:uniquelyconnected} A uniquely orientable graph is either connected and coconnected or is $P_m$ or $\overline{P_m}$, $m=2,3$. In particular, an irreducible uniquely orientable graph is connected and coconnected.
\end{lemma}
\begin{proof}
Without loss of generality, suppose that a uniquely orientable graph $X$ 
is not connected. Then $\overline{X}$ is connected and it is a composition graph 
with two modules, namely, the vertex set of any connected component of $X$ 
and its complement in~$\Omega$. By \cite[Theorem 5.12]{G}, every partitive set of 
a UPO graph induces an empty graph. It follows that $X=K_n\cup K_m$ for some integers $n,m$. 
Since $X$ is a UPO graph, it follows that $n\cdot m\leq 2$, and we are done.   
\end{proof}

For a graph $X=(\Omega,E)$ and $\mathbf{e}\in E\cup \overline{E}$, 
let us denote by $\Omega(\mathbf{e})$ the set of vertices incident 
to at least one element from $I(\mathbf{e})$, and by $X(\mathbf{e})$ 
the subgraph induced by $\Omega(\mathbf{e})$ in $X$ or in $\overline{X}$
depending on whether $\mathbf{e}\in E$ or $\mathbf{e}\in \overline{E}$.

\begin{corollary}\label{coro:Omega}
Let $X=(\Omega,E)$ be an irreducible uniquely orientable graph. 
Then $\Omega=\Omega(\mathbf{e})$ for every $\mathbf{e}\in E\cup \overline{E}$.
\end{corollary}
\begin{proof}
Since $P_m$ and $\overline{P_m}$, $m=2,3$, are not irreducible, 
$X$ and $\overline{X}$ are connected by Lemma \ref{lm:uniquelyconnected}. 
The lemma then follows from Eq. \eqref{eq:uniquely}. 
\end{proof}

In the rest of this section, we show that an irreducible permutation graph 
is a certain composition graph; this observation was implicitly used in \cite{C1981}. 
Given a graph $X=(\Omega,E)$, define a binary relation $\sim$ on $\Omega$ 
as follows: $\alpha\sim\beta$ if and only if $\alpha=\beta$ or 
there is an edge $\mathbf{e}\in E$ 
such that the graph $X(\mathbf{e})$ is uniquely orientable and $\alpha,\beta\in \Omega(\mathbf{e})$.

\begin{lemma}\label{lm:Colbourn}
Let $X$ be an irreducible permutation graph that is not uniquely orientable. 
Then the relation $\sim$ defined by $X$ is a non-trivial equivalence relation, 
the classes of which induce uniquely orientable graphs. 
\end{lemma}
\begin{proof}
It follows from \cite[Lemma~6]{C1981} that whenever, for $\mathbf{e},\mathbf{e}'\in E$, 
the subgraphs $X(\mathbf{e})$ and $X(\mathbf{e}')$ are uniquely orientable, 
the sets $\Omega(\mathbf{e})$ and $\Omega(\mathbf{e}')$ are either equal or disjoint. 
Therefore, $\sim$ is an equivalence relation whose 
equivalence classes induce uniquely orientable graphs. Moreover, 
by \cite[Lemma~5]{C1981} and the fact that $X$ is irreducible, there exist  
an edge $\mathbf{e}$ such that the subgraph $X(\mathbf{e})$ is uniquely
orientable. 
Since $X$ is not uniquely orientable, $X\ne X(\mathbf{e})$ 
and we are done.
\end{proof}

\begin{lemma}\label{lm:interspace}
In the notation of Lemma \ref{lm:Colbourn}, 
$X$ is a non-trivial composition graph with respect 
to the equivalence relation $\sim$.  
\end{lemma}
\begin{proof} 
It suffices to show that every $\Delta\in \Omega/{\sim}$ is a partitive set of $X$. 
This is obviously true if $|\Delta|=1$. If $|\Delta|>1$, then, by the definition of $\sim$, 
there exists $\mathbf{e}\in E$ such that $\Delta=\Omega(\mathbf{e})$. 
It follows from \cite[Lemma~3]{C1981} that the set $\Omega(\mathbf{e})$ is partitive. 
Therefore, $X$ a composition graph with respect to the equivalence relation $\sim$. 
By Lemma \ref{lm:Colbourn} this relation is non-trivial, hence 
so is the composition graph.
\end{proof}

\begin{remark}
Every module of the composition graph in Lemma \ref{lm:interspace} 
either induces a uniquely orientable non-empty subgraph of $X$ or is a singleton.  
\end{remark}

Note that the composition graph defined in Lemma \ref{lm:interspace} 
is minimal with respect to $|\cC|$ among all composition graphs $X=X_0[\cC]$ 
where $X_0$ is a permutation graph and $\cC$ is a collection of uniquely 
orientable subgraphs of $X$.

\begin{lemma}{\rm (cf. \cite[Theorem~5.8]{G})}\label{lm:betweenmodules} 
Let $X$ be the composition graph as in Lemma \ref{lm:interspace}. 
Then for all distinct modules $\Delta,\Delta'$ and 
every $\mathbf{e}\in \Delta\times \Delta'$, 
    the implication class $I(\mathbf{e})$ intersects neither 
    $\Delta\times \Delta$ nor $\Delta'\times \Delta'$, and, moreover,  
$\Delta\times \Delta'\subseteq I(\mathbf{e})$ and 
$\Delta'\times\Delta\subseteq I(\mathbf{e})^*$.
\end{lemma}
\begin{proof}
Assume on the contrary that $I(\mathbf{e})$ contains, say, $\mathbf{f}\in \Delta\times \Delta$. Then $I(\mathbf{f})\subseteq \Delta\times \Delta$ 
by the definition of $\sim$ and Eq. \eqref{eq:uniquely}. 
Since two implication classes either coincide or are disjoint, 
we obtain $I(\mathbf{e})=I\left(\mathbf{f}\right)\subseteq \Delta\times \Delta$, 
a contradiction.

Further, let $\mathbf{e}=(\alpha,\beta)$ and $\mathbf{e}\in E$ 
(the case $\mathbf{e}\in \overline{E}$ is similar). 
Then $\{\alpha\}\times \Delta'\subseteq E_{\Delta,\Delta'}$. 
Note the complement of $X_{\Delta'}$ is connected by Lemma \ref{lm:uniquelyconnected}, 
since the graph $X$ is irreducible and so is $X_{\Delta'}$. 
Therefore, for every $\beta'\in \Delta'$, there is a path from 
$\beta$ to $\beta'$ in the complement of $X_{\Delta'}$. 
This implies that $(\alpha,\beta')\in I(\mathbf{e})$ 
and hence we obtain $\{\alpha\}\times \Delta'\subseteq I(\mathbf{e})$. 
Similarly, one can see that 
$\Delta\times \{\beta\} \subseteq I(\mathbf{e})$. 
Thus, the implication class $I(\mathbf{e})$ 
intersects non-trivially the implication class 
$I(\mathbf{e'})$ of every other edge $\mathbf{e'}\in E_{\Delta,\Delta'}$, 
and hence $I(\mathbf{e})=I(\mathbf{e'})$. 
This implies $E_{\Delta,\Delta'}\subseteq I(\mathbf{e})$ and, by symmetry, 
$E_{\Delta',\Delta}\subseteq I(\mathbf{e})^*$.
%
\end{proof}

\section{Coherent configurations and their extensions}\label{sect:cc&extensions}
In this section we provide a short background of the theory of coherent configurations. 
We mainly follow the notation and terminology from~\cite{CP}, 
where further details and all unexplained facts can be found.

\subsection{Basic definitions} Let $\Omega$ be a finite set and $S$ a partition of $\Omega^2$; in particular, the elements of $S$ are treated as binary relations on~$\Omega$. A pair $\mathcal{X}=(\Omega,S)$ is called a {\bf coherent configuration} on $\Omega$ if the following conditions are satisfied:
\nmrt
\tm{C1}  the diagonal relation $1_{\Omega}$ 
is a union of some relations of~$S$,
\tm{C2} for each $s\in S$, the relation $s^*$ belongs to~$S$,
\tm{C3} given $r,s,t\in S$, the number $c_{rs}^t=|\alpha r\cap \beta s^{*}|$ does not depend on $(\alpha,\beta)\in t$. 
\enmrt

In what follows, we write $S=S(\cX)$; any relation belonging to $S$ (respectively, $S^\cup$) 
is called a {\bf basis relation} (respectively, a {\bf relation} of~$\cX$). 
A set $\Delta \subseteq \Omega$ is called a {\bf fiber} of $\cX$ if the relation $1_{\Delta}$
is basis. 
The set of all fibers is denoted by $F=F(\cX)$. Any element of $F^\cup$ is called 
a {\bf homogeneity set} of $\cX$. In particular, the (left, right) support of every 
relation of $\cX$ is a homogeneity set. 

\subsection{Isomorphisms and separability} Let $\cX=(\Omega,S)$ and $\cX'=(\Omega',S')$ be two coherent configurations. A bijection $f\colon\Omega\to\Omega'$ is called a (combinatorial) {\bf isomorphism} from $\cX$ to $\cX'$ if the relation $s^f$
belongs to $S'$ for every $s\in S$. The combinatorial isomorphism $f$ induces a natural bijection $\varphi\colon S\to S'$, $s\mapsto s^f$. One can see that $\varphi$  preserves the numbers from the condition~(C3), namely, the numbers $c_{rs}^t$ and $c_{r^{\varphi},s^{\varphi}}^{t^{\varphi}}$ are equal  for all $r,s,t\in S$. Every bijection $\varphi\colon S\to S'$ having this property is called an {\bf algebraic isomorphism}, written as $\varphi\colon \cX\to \cX'$. A coherent configuration is called {\bf separable} if every algebraic isomorphism from it to another coherent configuration is induced by an isomorphism. 

An algebraic isomorphism $\varphi\colon \cX\to\cX'$ induces a uniquely determined bijection $S^\cup\to {S'}^\cup$ denoted also by $\varphi$. For any $\Delta\in F^\cup$, we have $\varphi(1_\Delta)=1_{\Delta^\varphi}$ for a uniquely determined $\Delta^\varphi\in {F'}^\cup$. The induced mapping $\Delta\mapsto\Delta^\varphi$ defines a bijection $F^\cup \to {F'}^\cup$ that takes $F$ to $F'$.

\subsection{Parabolics.}\label{ssect:parabolics}
An equivalence relation $e$ on a set $\Delta\subseteq\Omega$ is called a {\bf partial parabolic} of the coherent configuration~$\cX$ if $e$ is the union of some basis relations; if, in addition, $\Delta=\Omega$, then $e$ is called a {\bf parabolic} of~$\cX$. Note that the transitive closure of any symmetric relation of $\cX$ is a partial parabolic.

A (partial) parabolic $e$ is said to be {\bf decomposable} if $e$ is the  union of pairwise disjoint non-empty partial parabolics; we say that $e$ is {\bf indecomposable} if it is not decomposable. Every partial parabolic is the disjoint union of uniquely determined indecomposable partial parabolics; they are called the {\bf indecomposable components} of~$e$.

Let $e$ be a partial parabolic, and let $\Delta\in\Omega/e$. Denote by $S_{\Omega/e}$ and $S_\Delta$ the sets of all non-empty relations 
$$
s\sbo{\Omega/e}=\{(\Delta,\Gamma)\in\Omega/e\times\Omega/e\colon\ s\sbs{\Delta}{\Gamma}\ne\varnothing\}\quad\text{and}\quad s\sbo{\Delta}=s\sbo{\Delta,\Delta},
$$
respectively, where $s\in S$. Then the pairs $\cX_{\Omega/e}=(\Omega/e,S_{\Omega/e})$ and $\cX_\Delta=(\Delta,S_\Delta)$ are coherent configurations called the {\bf quotient} of~$\cX$ modulo~$e$ and {\bf restriction} of~$\cX$ to~$\Delta$. Note that when $e=1_\Delta$ for a homogeneity set $\Delta$, we have $\cX_{\Omega/e}=\cX_\Delta$.

Let $e^{\circ}$ be an indecomposable component of $e$. 
Then $e^{\circ}_{\Omega/e}$ is a reflexive relation of $\cX_{\Omega/e}$, 
that is, $e^{\circ}_{\Omega/e}=1_{\Delta^{\circ}}$ where 
$\Delta^{\circ}$ is a homogeneity set of $\cX_{\Omega/e}$. 
Let $\Pi(e)$ denote the set of all $\Delta^{\circ}$ as 
$e^{\circ}$ runs over the set of all indecomposable component of $e$. 
Note that $\Pi(e)$ is a partition of $\Omega/e$, while the corresponding 
equivalence relation is a parabolic of $\cX_{\Omega/e}$.

Any algebraic isomorphism $\varphi\colon \cX\to\cX'$ induces a bijection between partial parabolics of $\cX$ and $\cX'$ that preserves the property of a partial parabolic to be indecomposable. Let  $e$ be a partial parabolic of $\cX$ and $e'=\varphi(e)$. Then the mapping $s\sbo{\Omega/e}\mapsto \varphi(s)\sbo{\Omega'/e'}$, $s\in S$, defines an algebraic isomorphism $\varphi_{\Omega/e}\colon\cX_{\Omega/e}\to\cX'_{\Omega'/e'}$. 

We say that the classes $\Delta\in\Omega/e$ and $\Delta'\in \Omega'/e'$ are $\varphi$-{\bf associated} if $\varphi$ takes the indecomposable component of $e$ 
containing $\Delta$ as a class to the indecomposable component of $e'$ 
containing $\Delta'$ as a class.
According to \cite[Example~2.3.16]{CP}, in this case $\varphi$ induces 
an algebraic isomorphism 
$\varphi_{\Delta,\Delta'}\colon \cX^{}_{\Delta}\to \cX'^{}_{\Delta'}$ 
such that $\varphi_{\Delta,\Delta'}(s_{\Delta})=\varphi(s)_{\Delta'}$ 
for every $s\in S$.

\subsection{Coherent closure}
There is a natural partial order\, $\le$\, on the set of all coherent configurations 
on the same set~$\Omega$. Namely, given two such coherent configurations $\cX$ and 
$\cX'$, we set $\cX\le\cX'$ if and only if each basis relation of~$\cX$ is the union 
of some basis relations of~$\cX'$. In other words, this means that the partition of 
$\Omega^2$ into basis relations of $\cX'$ is finer than the partition of $\Omega^2$ 
into basis relations of~$\cX$. The minimal and maximal elements with respect to this 
ordering are the {\bf trivial} and {\bf discrete} coherent configurations: the basis 
relations of the former one are the reflexive relation $1_\Omega$ and its complement 
in $\Omega\times\Omega$ (if $|\Omega|\geq 1$), whereas the basis relations of 
the latter one are singletons. 
Note that the trivial and discrete coherent configurations are separable.

\begin{lemma}\label{lemma-ccturnir}
A coherent configuration on~$\Omega$ having a transitive tournament on $\Omega$ 
as a relation is discrete.
\end{lemma}
\prf
Let $s$ be a relation of a coherent configuration on~$\Omega$. Suppose that 
the supports of $s$ equal $\Omega$ and $s$ is a transitive tournament. 
Then there are no two distinct points $\alpha,\beta$ such that $|\alpha s|=|\beta s|$. 
Hence, no two distinct points lie in the same fiber, which means that the coherent configuration is discrete.
\eprf

The {\bf coherent closure} $\WL(T)$ of a set $T$ of binary relations on $\Omega$, is defined 
to be the smallest coherent configuration on $\Omega$ such that each relation of~$T$ is a 
union of some basis relations. 
The coherent closure is canonical with respect to algebraic isomorphisms 
in the sense that if $\varphi,\psi\colon \wl{(T)}\to \wl{(T')}$ are 
algebraic isomorphisms such that $\varphi(t)=\psi(t)$ for all $t\in T$, 
then $\varphi=\psi$.
Furthermore, the coherent closure is a closure operator\footnote{with respect to the natural 
partial order on the partitions of the same set.} on the set of all partitions of $\Omega^2$ 
satisfying conditions (C1) and (C2) in the definition of a coherent configuration. 

For points $\alpha,\beta,\ldots\in\Omega$, we denote by $\cX_{\alpha,\beta,\ldots}$ the coherent closure of the union of $S$ and the set $\{1_{\{\alpha\}},1_{\{\beta\}},\ldots\}$. 

For an equivalence relation $e$ on $\Omega$, we denote by $\cX_e$ the coherent closure 
of the union of $S$ and $\{e\}$. 

For a partition $\pi$ of $\Omega$, we denote by $\cX_\pi$ the coherent closure of the union of $S$ and all of $1_\Delta$, $\Delta\in\pi$. 

The {\bf coherent configuration of a graph} $X$ is defined to be 
the coherent closure of its edge set: $\WL(X)=\WL(\{E(X)\})$.
Note that $\WL(X)=\WL(\overline{X})$.

\begin{lemma}\label{lm:twinparabolic}
The $0$-equivalence and $1$-equivalence of a graph are 
parabolics of its coherent configuration.
\end{lemma}
\prf Follows from \cite[Proposition 4.10]{GNP}.\eprf

\subsection{Direct sum and tensor product.}
Let $\cX=(\Omega,S)$ and $\cX'=(\Omega',S')$ be two coherent configurations. Denote by $\Omega\sqcup\Omega'$ the disjoint union of~$\Omega$ and~$\Omega'$, and by $S\boxplus S'$ the union of the set $S\sqcup S'$ and the set of all relations $\Delta\times\Delta'$ and $\Delta'\times\Delta$ with $\Delta\in F$ and $\Delta'\in F'$. Then the pair
$$
\cX\boxplus\cX'=(\Omega\sqcup\Omega',S\boxplus S')
$$
is a coherent configuration called the {\bf direct sum} of~$\cX$ and~$\cX'$. One can see that $\cX\boxplus\cX'$ is the smallest coherent configuration $\cY$ on~$\Omega\sqcup\Omega'$ such that 
$\cX=\cY_{\Omega}$ and $\cX'=\cY_{\Omega'}$.
It should be noted that $\cX\boxplus\cX'$ is separable if and only if so are $\cX$ and~$\cX'$, 
see \cite[Corollary~3.2.8]{CP}. 

\begin{lemma}\label{lm:directsum}
Let $X$ be a composition graph and $\pi$ be the partition of the vertex set of $X$ 
into the modules of this composition. 
Then $\wl{(X)}_\pi=\boxplus_{\Delta\in\pi}\wl{(X_{\Delta})}$.
\end{lemma}
\begin{proof}
Let $X=(\Omega,E)$. Then, given distinct $\Delta,\Gamma\in \pi$, 
we have 
$$
E_{\Delta,\Gamma}=\varnothing\ \text{ or }\ \Delta\times \Gamma\subseteq E.
$$
Let $X'=(\Omega,E')$ be the graph obtained from $X$ by removing 
the set $E_0$ of all edges $E_{\Delta,\Gamma}$, $\Delta\ne\Gamma$. 
Then the vertex set of every connected component of $X'$ is contained 
in some $\Delta\in \pi$. Moreover, 
$E'=E\setminus E_0$ and $E=E'\cup E_0$. Since $E_0$ is a relation 
of both $\wl{(X')}_{\pi}$ and $\wl{(X)}_{\pi}$, we obtain 
\[
\wl{(X')}_{\pi}=\wl{(X)}_{\pi}.
\]
It is easily seen that each graph $X_\Delta$ is a union of 
connected components of the graph $X'$, so that the lemma follows 
from \cite[Exercise~3.7.35]{CP}.
\end{proof}

Given coherent configurations $\cX_1=(\Omega_1,S_1)$ and $\cX_2=(\Omega_2,S_2)$ denote by $S_1\otimes S_2$
the set of all relations 
$$
s_1\otimes s_2=\left\{((\alpha_1,\alpha_2),(\beta_1,\beta_2))\in (\Omega_1\times\Omega_2)^2:\ (\alpha_1,\beta_1)\in s_1,\ (\alpha_2,\beta_2)\in s_2\right\},
$$
where $s_1\in S_1$ and  $s_2\in S_2$. Then the pair
$\cX_1\otimes\cX_2=(\Omega_1\times \Omega_2,S_1\otimes S_2)$ is a coherent configuration. It is called the
{\bf tensor product} of~$\cX_1$ and~$\cX_2$. For each positive integer~$m$, the $m$-tensor power of $\cX$ is denoted by $\cX^m$.


\subsection{Extensions of coherent configurations and their algebraic isomorphisms}\label{011122a}

Let $\cX$ be a coherent configuration on $\Omega$ and $m$ a positive integer. The {\bf $m$-extension} of $\cX$ is
by definition the following coherent configuration on $\Omega^m$:
\[
\wh\cX^{(m)}=\WL\left(S\left(\cX^m\right)\cup \{1_{\diag(\Omega^m)}\}\right)
\]
where $\cX^m$ is the $m$-fold tensor power of $\cX$. 
We observe that except for trivial cases the $m$-extension is a non-homogeneous coherent configuration for all $m\ge 2$. From the definition it follows that $\wh\cX^{(1)}=\cX$ and
$\wh\cX^{(m)}\geq \cX^{m}$.

Let $\varphi\colon\cX\to\cX'$ be an algebraic isomorphism. 
Then $\varphi$ induces the component-wise algebraic isomorphism $\varphi^m:\cX^m\to{\cX'}^m$. 
An algebraic isomorphism $\psi\colon\wh\cX^{(m)}\to\wh{\cX}'^{(m)}$ is called an {\bf $m$-extension} if
\nmrt
\tm{1} $(\diag(\Omega^m))^\psi=\diag((\Omega')^m)$,
\tm{2} $s^\psi=s^{\varphi^m}$ for all $s\in S(\cX^m)$.
\enmrt
Each algebraic isomorphism obviously has a 1-extension, which always 
coincides with it. Furthermore, for any~$m$ the existence of the $m$-extension 
of $\varphi$ implies its uniqueness (this follows from the canonicity 
of the coherent closure with respect to algebraic isomorphisms); 
we denote it by $\wh\varphi^{(m)}$. 
We note that not every algebraic isomorphism has an $m$-extension. 

An algebraic isomorphism is called an {\bf $m$-isomorphism} if it has the $m$-extension. 
Every algebraic isomorphism induced by some isomorphism is 
an $m$-isomorphism for all~$m$. 
Note that every $m$-isomorphism is a $k$-isomorphism for all $k\le m$.

The coherent configuration 
\[
\ov\cX=\ov\cX^{(m)}=((\wh\cX^{(m)})_{\diag(\Omega^m)})^{\partial_m^{-1}},
\]
where $\partial_m\colon\alpha\mapsto (\alpha,\ldots,\alpha)$ is the diagonal mapping from $\Omega$ to~$\Omega^m$, is called the {\bf $m$-closure} of~$\cX$; 
if $\cX={\ov\cX}^{(m)}$, then $\cX$ is
said to be {\bf $m$-closed} (in particular, any coherent configuration is $1$-closed). 
Any $m$-isomorphism~$\varphi$ uniquely extends to the algebraic isomorphism $\ov\varphi^{(m)}$ between the corresponding $m$-closures. 

A coherent configuration $\cX$ is said to be {\bf $m$-separable} if, for all $\cX'$, 
every $m$-isomorphism from $\cX$ to $\cX'$ is induced by an isomorphism. 
The integer
$$
s(\cX)=\min\{m:\ \cX\ \text{is $m$-separable}\}
$$
is called the {\bf separability number} of~$\cX$.
Obviously, $s(\cX)=1$ if and only if a coherent configuration $\cX$ is separable.


We will frequently use throughout the paper the following lemma (see \cite[Lemma~6.2]{EP2000}). 
Below, given $s\subseteq \Omega^2$ and $i,j\in\{1,\ldots,m\}$, 
we set
$$
\cyl_s(i,j)=\{(x,y)\in \Omega^{m}\times \Omega^{m}\colon\ (x_i,y_j)\in s\}.
$$

\begin{lemma}\label{lm:cyl}
Let $\cX$ be a coherent configuration on $\Omega$ and $m$ a positive integer. 
For all $s\in S(\cX)^\cup$ and $i,j\in\{1,\ldots,m\}$,  
one has:
\nmrt
\tm{i} $\cyl_s(i,j)$ is a relation of the coherent configuration $\wh\cX^{(m)}$;
\tm{ii} $\cyl_s(i,j)^{\wh\varphi}=\cyl_{s^{\ov\varphi}}(i,j)$ for any $m$-isomorphism $\varphi$,
where $\wh\varphi=\wh\varphi^{(m)}$ and $\ov\varphi=\ov\varphi^{(m)}$.
\enmrt
\end{lemma}

\section{Composition graphs and their coherent configurations}\label{011122c}

In this section we study the coherent configurations of composition graphs 
and their algebraic isomorphisms. Throughout the section, let $X=(\Omega,E)$ 
be a composition graph with respect to an equivalence relation $e$ on $\Omega$, 
$X_0$ the quotient graph of $X$ modulo $e$, and $\pi=\Omega/e$.  

\begin{lemma}\label{lm:wlY}
In the above notation, one has  
\[
\wl{(X)}_{\pi}=\boxplus_{\Delta\in \pi}\wl{(X_{\Delta})}\quad\mbox{and}\quad 
\wl{(X)}_{\pi}\geq \wl{(X)}_e.
\]
In particular, the restriction of $\wl{(X)}_e$ to any $\Delta\in \pi$ equals $\wl{(X_{\Delta})}$.
\end{lemma}
\begin{proof}
The first formula follows from Lemma \ref{lm:directsum}. 
Further, as $e$ is a parabolic of $\wl{(X)}_{\pi}$ and 
$\wl{(X)}_{\pi}\geq \wl{(X)}$, one has $\wl{(X)}_{\pi}\geq \wl{(X)}_e$ 
and hence:
\[
\wl{(X_{\Delta})}=\left(\wl{(X)}_{\pi}\right)_{\Delta}\geq \left(\wl{(X)}_e\right)_{\Delta}\geq \wl{(X_{\Delta})},\] 
whence the lemma follows.
\end{proof}

The next theorem involves the notation and definitions 
from Section \ref{ssect:parabolics}.

\begin{theorem}\label{theo:C}
In the above notation, 
let $\cX'$ be a coherent configuration on $\Omega'$, 
$\varphi\colon \wl{(X)_e}\to \cX'$ be an algebraic isomorphism, 
$e'=\varphi(e)$, $\pi'=\Omega'/e'$, and $E'=\varphi(E)$. 
Then the graph $X'=(\Omega',E')$ 
is a composition graph with respect to $e'$. Moreover, one has:
\begin{itemize}
    \item[(i)] $\cX'=\wl{(X')_{e'}}$;
    \item[(ii)] $\varphi_{\Delta,\Delta'}\left(\wl{(X_{\Delta})}\right)=\wl{(X'_{\Delta'})}$ for all $\varphi$-associated $\Delta\in \pi$ and $\Delta'\in \pi'$;
    \item[(iii)] $\varphi$ induces an algebraic isomorphism $\varphi_0\colon\wl(X_0)_{\Pi}\to\wl(X'_0)_{\Pi'}$, where $\Pi=\Pi(e)$ 
    and $\Pi'=\Pi(e')$.  
  \end{itemize}
\end{theorem}
\begin{proof} 
Recall that $E$ is a relation of the coherent configuration $\cX=\wl{(X)}_{e}$ 
and, by Lemma \ref{lm:eE}, one has $e\cdot (E-e)=(E-e)\cdot e=E-e$. 
Since algebraic isomorphism respect the set-theoretical operations with 
relations, we obtain $e'\cdot (E'-e')=(E'-e')\cdot e'=E'-e'$. 
Hence, by Lemma \ref{lm:eE}, $X'$ is a composition graph with respect to $e'$. 
Then (i) follows from the definition of $X'$, whereas (ii) follows from 
the last statement of Lemma \ref{lm:wlY}. 
Finally, the algebraic isomorphism $\varphi$ induces an algebraic isomorphism $\varphi_0\colon\cX_{\Omega/e}\to\cX'_{\Omega'/e'}$ which takes the edge set $E_{\Omega/e}$ of $X_0$ to the edge set $E'_{\Omega'/e'}$ of $X_0'$. 
Therefore, $\varphi_0(\WL(X_0))=\WL(X'_0)$. It remains 
to observe that $\varphi$ and hence $\varphi_0$ takes the classes of $\Pi$ 
to those of $\Pi'$.
\end{proof}

\begin{theorem}\label{theo:iso}
In the notation of Theorem \ref{theo:C}, the following holds:
\begin{itemize}
    \item[(i)] if $\varphi_0$ and each of $\varphi_{\Delta,\Delta'}$ are induced 
    by isomorphisms, then so is $\varphi$;
    \item[(ii)] if $\varphi$ is an $m$-isomorphism for some natural $m$, then so are $\varphi_0$ and each of $\varphi_{\Delta,\Delta'}$. 
\end{itemize}
\end{theorem}
\begin{proof}
Let $f_0\colon \pi\to \pi',\Delta\mapsto\Delta'$ be a bijection that 
induces $\varphi_0$. Further, let $f_{\Delta,\Delta'}\colon \Delta\to \Delta'$ 
be a bijection that induces $\varphi_{\Delta,\Delta'}$. 
Then there exists a uniquely determined bijection 
$f\colon \Omega\to\Omega'$ such that 
$f\mid_{\Delta}=f_{\Delta,\Delta'}$ for all $\Delta\in\pi$. 
We shall prove that $f$ induces $\varphi$, i.e., $s^f=\varphi(s)$ 
for all basis relations $s$ of the coherent configuration $\wl{(X)}_e$. 

Let us first assume that $s\cap e=\varnothing$. Then, by Lemma \ref{lm:wlY} 
and by the definition of a direct sum of coherent configurations,  
one can see that 
\[
s=\bigcup_{(\Delta_1,\Delta_2)\in s_{\Omega/e}} \Delta_1\times \Delta_2
\]
and then 
\[
s^f=\bigcup_{(\Delta_1',\Delta_2')\in (s_{\Omega/e})^{f_0}} \Delta_1'\times \Delta_2'.
\]
Since $f_0$ induces $\varphi_0$, it follows that $(s_{\Omega/e})^{f_0}=\varphi_0(s_{\Omega/e})=\varphi(s)_{\Omega'/\varphi(e)}=
s'_{\Omega'/e'}$, where $s'=\varphi(s)$. Thus, $s^f=\varphi(s)$.

Now let $s\subseteq e$. Denote by $e^{\circ}$ be the indecomposable component 
of $e$ that contains~$s$. Then $s=\bigcup_{\Delta\in \Lambda}s_{\Delta}$, 
where $\Lambda=\Omega/e^{\circ}$ is a class of $\Pi$. 
Define $\Lambda'=\Omega'/{\varphi(e^{\circ})}$, which is a class of $\Pi'$. 
Then:
\[
s^f = \bigcup_{\Delta^f\in \Lambda^f}\left(s_{\Delta}\right)^f = 
\bigcup_{\Delta'\in \Lambda'}\left(s_{\Delta}\right)^{f_{\Delta,\Delta'}} = 
\bigcup_{\Delta'\in \Lambda'}\varphi_{\Delta,\Delta'}\left(s_{\Delta}\right)=
\bigcup_{\Delta'\in \Lambda'}\varphi(s)_{\Delta'}=
\varphi(s).
\]

(ii) Put $\cX=\wl{(X)}_e$ and let $\cX'$ be as in Theorem \ref{theo:C}(i). 
Denote by $\wh{\cX}, \wh\cX'$ the $m$-extensions of $\cX$ and $\cX'$, 
respectively, and by $\wh{\varphi}\colon \wh\cX\to \wh\cX'$ 
the $m$-extension of $\varphi$. 
Then, for every $\Delta\in \pi$, we have $\wl(X_{\Delta})=\cX_{\Delta}$ by Lemma \ref{lm:wlY} and 
\[
\wh{\cX}_{\Delta^m}\ge \wh{\cX_{\Delta}}^{(m)}. 
\]
It follows from Lemma \ref{lm:cyl}(i) that 
$\cap_{1\leq i,j\leq m}\cyl_e(i,j)$ is a relation of $\wh\cX$. 
Therefore, the support of this relation, which equals 
$\Omega'=\cup_{\Delta\in \pi}\Delta^m$, 
is a homogeneity set of $\wh\cX$. Furthermore, since $e^m$ is a parabolic 
of $\cX^m$, it follows that $(e^m)_{\Omega'}$ is a partial parabolic of $\wh\cX$, 
whose classes are $\Delta^m$, $\Delta\in\pi$. 
Hence we can define an algebraic isomorphism $\wh{\psi}=\wh{\varphi}^{}_{\Delta^m}$ 
from $\wh\cX_{\Delta^m}$ to $\wh\cX'_{(\Delta')^m}$.
Since $\diag{(\Delta^m)}=\diag{(\Omega^m)}\cap \Delta^m$, 
one can see that $\wh{\psi}$ 
takes $\diag{(\Delta^m})$ to 
$\diag{((\Delta')^m)}$. Moreover, for a basis relation $s=s_1\otimes\cdots\otimes s_m$ of $\cX^m$, 
$$
\wh\psi(s_{\Delta^m})=\wh\varphi_{\Delta^m}((s_1)_\Delta\otimes\cdots\otimes (s_m)_\Delta)=\wh\varphi(s_1\otimes\cdots\otimes s_m)_{\Delta^m}=(\varphi_{\Delta,\Delta'})^m(s).
$$
Thus, $\wh{\psi
}$ is the $m$-extension of $\varphi_{\Delta,\Delta'}$.
\end{proof}

\section{Uniquely orientable graphs}\label{sect:uniq}

In this section, we obtain an upper bound on the separability number 
of uniquely orientable graphs. This bound follows from Theorem \ref{theo:uniq}. 
Its proof requires the lemma below, which will also be used in Section \ref{sect:main}. In what follows, let $X=(\Omega,E)$ be a graph, 
$\cX=\wl(X)$, $\wh\cX$ denote $\cX^{(2)}$, and 
$\partial=\partial_2\colon\alpha\mapsto (\alpha,\alpha)$ be 
the diagonal mapping from $\Omega$ to~$\Omega^2$. 



\begin{lemma}\label{lemma-Gammarel}
Let $X$ be an irreducible permutation graph. 
Denote by $e_1=e_1(X)$ and $e_2=e_2(X)$ partial equivalence relations 
on $\Omega^2$ such that
\[
\Omega^2/e_1=\{I(\mathbf{e})\mid \mathbf{e}\in E\}\quad\text{and}\quad 
\Omega^2/e_2=\{\Delta^2\setminus \diag{\left(\Delta^2\right)}\mid \Delta\in\Omega/{\sim}\}.
\]
Then $e_1,e_2$ are partial parabolics of $\wh{\cX}$.
\end{lemma}
\begin{proof} 
By Lemma \ref{lm:cyl}, the coherent configuration $\wh{\cX}$ has the following relations 
\begin{eqnarray*}
L_1&=&\cyl_{1_{\Omega}}(2,2)\cap \cyl_{\overline{E}}(1,1)\cap  
\cyl_{E}(1,2)\cap \cyl_{E}(2,1),\\
L_2&=&\cyl_{1_{\Omega}}(1,1)\cap \cyl_{\overline{E}}(2,2)\cap  
\cyl_{E}(1,2)\cap \cyl_{E}(2,1).
\end{eqnarray*}
Hence $\wh\cX$ has the relation $L_1\cup L_2$, which is exactly 
the $\Gamma$-relation on the edges of $X$. 
Therefore, $e_1$ coincides with the transitive closure of $L_1\cup L_2$, 
which is a partial parabolic of $\wh{\cX}$.

Let us prove that $e_2$ is a partial parabolic of $\wh\cX$. 
Define a partial parabolic $e_0=e_1(X)\cup e_1(\overline{X}){\cup 1_{\diag(\Omega^2)}}$ of $\wh\cX$. 
Observe that by Eq. \eqref{eq:uniquely} every class $\Delta^2\setminus \diag{\left(\Delta^2\right)}$ 
of $e_2$, $\Delta\in \Omega/{\sim}$, is the union of the four implication classes of $X_{\Delta}$, 
which are classes of $e_0$. In order to identify the classes of $e_2$, let us define 
an auxiliary relation
\begin{equation*}
s:=\left(\cyl_{1_{\Omega}}(1,1)\cup \cyl_{1_{\Omega}}(2,2)\right)\cap \left(\left(E\cup \overline{E}\right)\times \diag{\left(\Omega^2\right)}\right)
\end{equation*}
Observe that $s$ is a relation of $\wh{\cX}$. Indeed, 
$E$ coincides with the support of $\cyl_{1_\Omega}(1,1)\cap \cyl_{E}(1,2)\cap \cyl_{1_\Omega}(2,2)$, 
which is a relation of $\wh{\cX}$ by Lemma \ref{lm:cyl}.
Hence $E$ (and similarly $\overline{E}$) is a homogeneity set of $\wh{\cX}$, 
whence $s\in S(\cX)^\cup$. 

Note that $s$ consists of all pairs $((\alpha,\beta),(\gamma,\gamma))$ 
where $(\alpha,\beta)\in E\cup \overline{E}$ and $\gamma\in \{\alpha,\beta\}$. 
Put $N(\mathbf{e}):=\cup_{\mathbf{f}\in I(\mathbf{e})}\mathbf{f}s$ 
(here we recall that $\mathbf{f}s$ means the neighborhood of a point $\mathbf{f}$ 
in the relation $s$). Then, for every $\mathbf{e}\in E\cup \overline{E}$, 
it follows from the definition of $\Omega(\mathbf{e})$ 
(see the paragraph before Corollary \ref{coro:Omega}) that 
\begin{equation*}
N(\mathbf{e}) = \Omega(\mathbf{e})^{\partial}.   
\end{equation*}
Suppose that $\Delta\in \Omega/{\sim}$ and $\mathbf{e}\in E\cup \overline{E}$.  
If $\mathbf{e}\in \Delta\times \Delta'$ for some $\Delta'\in \Omega/{\sim}$, 
$\Delta\ne \Delta'$, then $\Delta^\partial,\Delta'^\partial\subsetneq N(\mathbf{e})$ 
by Lemma \ref{lm:betweenmodules}. On the other hand, 
if $\mathbf{e}\in \Delta^2\setminus \diag{\left(\Delta^2\right)}$, then 
$N(\mathbf{e}) = \Omega(\mathbf{e})^{\partial} = \Delta^{\partial}$ by Corollary \ref{coro:Omega};
hence, given $\mathbf{f}\in E\cup \overline{E}$, one has 
\begin{equation}\label{Ne=Nf}
    N(\mathbf{e})=N(\mathbf{f})\quad\text{if and only if}\quad \mathbf{f}\in \Delta^2\setminus \diag{\left(\Delta^2\right)}.    
\end{equation}

Now, if $\rho$ denotes the natural surjection $\rho\colon\Omega^2\to\Omega^2/e_{0}$, then:
\begin{itemize}
    \item the $\rho$-image of $\diag{\left(\Omega^2\right)}$ is naturally identified with $\diag{\left(\Omega^2\right)}$ itself;
    \item the $\rho$-image of $I(\mathbf{e})$ is a singleton for every $\mathbf{e}\in E\cup \overline{E}$;
    \item the set of all pairs $\rho(I(\mathbf{e}))$ and $\rho(I(\mathbf{f}))$ such that 
$\rho(N(\mathbf{e}))=\rho(N(\mathbf{f}))$ is an equivalence relation, say $e'$, 
on $\rho\left(E\cup \overline{E}\right)$.
\end{itemize}

Further, since $s$ is a relation of $\wh\cX$, $\rho(s)$ is a relation of 
the quotient of the coherent configuration $\wh{\cX}$ modulo $e_{0}$.
Therefore, according to \cite[Exercise~2.7.8(1)]{CP}, we see that 
the equivalence relation $e'$ is a partial parabolic of this quotient.
It follows that $e^*=\rho^{-1}(e')$ is a partial parabolic of $\wh{\cX}$ 
(see \cite[Theorem~3.1.11]{CP}). 
Clearly, every class of $e^*$ is a union of some classes of $e_{0}$. 

Every class of $e_2$ that intersects some class $\Lambda$ of $e^*$ must coincide with $\Lambda$ 
by virtue of Eq. \eqref{Ne=Nf}. To complete the proof, put $\ov e=e^*\setminus e_2$ and 
define a binary relation 
\[
s_1=\{\left((\alpha,\beta),(\alpha',\beta')\right)\in (E\cup \overline{E})^2:\ 
\left((\alpha,\alpha'),(\alpha,\beta)\right)\in \ov{e}
\},
\]
It is easily seen that $s_1=(\overline{e}\cap \cyl_{1_{\Omega}}(1,1))\cdot \cyl_{1_{\Omega}}(2,1)$ and hence $s_1$ is a relation of~$\wh{\cX}$. 
By the above, a class $\Lambda$  of $e^*$ belongs to $e_2$ if and only if 
$s_1\cap \Lambda^2=\varnothing$.
Consequently,
$e_2=e^*\setminus \left(e^*\cdot \left(s_1\cap e^*\right)\cdot e^*\right)$, 
which implies that $e_2$ is a partial parabolic of $\wh{\cX}$.
\end{proof}

\begin{theorem}\label{theo:uniq}
Let $X$ be a uniquely orientable graph. 
Then for all $\mathbf{e}\in E$ and $\mathbf{f}\in \overline{E}$ 
the coherent configuration $\wh{\cX}_{\mathbf{e},\mathbf{f}}$ is discrete.
\end{theorem}
\begin{proof}
Choose an arbitrary $\mathbf{e}\in E$ and show that 
$I(\mathbf{e})^{\partial}=\{(\alpha^{\partial},\beta^{\partial})\mid (\alpha,\beta)
\in I(\mathbf{e})\}$ is a relation of $\wh{\cX}_{\mathbf{e}}$ 
(the case $\mathbf{e}\in \overline{E}$ is similar, as the graph $\overline{X}$ 
is also uniquely orientable). 
Since $I(\mathbf{e})$ and $I(\mathbf{e})^*$ are the only transitive orientations 
of $X$, the partial parabolic $e_1$ of $\wh\cX$, defined in Lemma \ref{lemma-Gammarel}, has the set of classes $\Omega^2/e_1=\{I(\mathbf{e}),I(\mathbf{e})^*\}$.
It follows that $I(\mathbf{e})$, which is the neighborhood of $\mathbf{e}$ in $e_1$, 
is a homogeneity set of $\wh{\cX}_{\mathbf{e}}$ (see \cite[Lemma~3.3.5]{CP}). 
Since $\Delta:=\diag{(\Omega^2)}$ is a homogeneity set of $\wh{\cX}$, 
\[
s_1=(\Delta\times I(\mathbf{e}))\cap \cyl_{1_{\Omega}}(1,1)\qaq s_2=(I(\mathbf{e})\times \Delta)\cap \cyl_{1_{\Omega}}(2,2)
\]
are relations of $\wh{\cX}_{\mathbf{e}}$.
Hence $s_1\cdot s_2=\{(\alpha^{\partial},\beta^{\partial})\mid (\alpha,\beta)
\in I(\mathbf{e})\}$ is a relation of $\wh{\cX}_{\mathbf{e}}$, as required.

Let $\mathbf{e}\in E$ and $\mathbf{f}\in \overline{E}$. 
By the above, $I(\mathbf{e})^{\partial}\cup I(\mathbf{f})^{\partial}$ is 
a relation of $\wh{\cX}_{\mathbf{e},\mathbf{f}}$, as the set of relations of 
$\wh{\cX}_{\mathbf{e},\mathbf{f}}$ includes the relations of both 
$\wh{\cX}_{\mathbf{e}}$ and $\wh{\cX}_{\mathbf{f}}$.
Furthermore, by Lemma \ref{lemma-tournament}, it is a transitive tournament 
on $\Delta$. Therefore, by Lemma \ref{lemma-ccturnir}, the restriction of $\wh{\cX}_{\mathbf{e},\mathbf{f}}$ to  $\Delta$ 
is discrete, and so is $\wh{\cX}_{\mathbf{e},\mathbf{f}}$. 
The theorem is proved.
\end{proof}

\begin{prop}\label{prop:uniq}
The separability number of a uniquely orientable graph is at most $6$.
\end{prop}
\begin{proof}
Let $X$ be a uniquely orientable graph and $\cX=\wl{(X)}$. By Theorem \ref{theo:uniq}, a $2$-point extension of $\wh{\cX}$
is a discrete coherent configuration. Hence, by Theorem \cite[Theorem~4.6(1)]{EP2000}, the separability number of $\wh{\cX}$ is at most $3$. 
Finally, by Theorem \cite[Theorem~4.6(3)]{EP2000}, 
the separability number of ${\cX}$ is at most $2\cdot 3=6$. 
\end{proof}

\section{Proof of Theorem \ref{thm:perm}}\label{sect:main}

We follow the notation from Section \ref{sect:uniq}. 
Let $X$ be a permutation graph. 
By \cite[Corollary~4.6.24]{CP}, 
it suffices to verify that the separability number of the coherent configuration 
$\cX=\wl{(X)}$ is at most 6. If $X$ is uniquely orientable, then the result 
follows from Proposition \ref{prop:uniq}. 
Suppose that $X$ is not uniquely orientable. Let $\psi\colon \cX\to \cX'$ be a 6-isomorphism. 
Then $\psi$ has the $2$-extension $\wh{\psi}$. 
We will need the following lemma.

\begin{lemma}\label{lm:parabolic}
Assume that $X$ is not uniquely orientable. Then it is 
a non-trivial composition graph with respect to an equivalence relation $e$ 
such that $e^{\partial}$ is a partial parabolic of 
the coherent configuration $\wh{\cX}$.
\end{lemma}
\begin{proof} 
If $X$ is reducible, the result follows from Lemmas \ref{lm:reducible} (with 
$e$ being either the $0$- or $1$-equivalence) and \ref{lm:twinparabolic}. 
Hence, in what follows, we assume that $X$ is irreducible.
Then $X$ is a non-trivial composition graph as defined in Lemma \ref{lm:interspace}, 
and let $e$ denote the equivalence relation $\sim$ (see the paragraph 
before Lemma \ref{lm:Colbourn}). 

Let $e_2$ denote the partial parabolic of $\wh{\cX}$ as defined 
in Lemma \ref{lemma-Gammarel}, and $\Lambda$ the support of $e_2$.
Denote by $s$ a relation on $\Lambda\times \diag{(\Omega^2)}$ 
such that two points are in $s$ if their first or second coordinates are equal, i.e., 
\[
s=
\left(\Lambda\times \diag{(\Omega^2)}\right)\cap 
\left(\cyl_{1_{\Omega}}(1,1)\cup \cyl_{1_{\Omega}}(2,2)\right).
\]
Then the transitive closure $t$ of $s^*\cdot s$ (which is obviously reflexive and symmetric) 
is a partial parabolic of $\wh{\cX}_{\diag{(\Omega^2)}}$. 
By the definition of $e_2$, the classes of $t$ coincide with $\Delta^{\partial}$, 
where $\Delta$ runs over the non-singleton classes of $e$. 
Thus, $e^\partial=t\cup 1_{\diag{(\Omega^2)\setminus T}}$, where 
$T$ is the support of $t$. Since $T$ is a homogeneity set, $e^\partial$ 
is a partial parabolic of $\wh\cX$. The lemma is proved.
\end{proof}

Let $e$ be the equivalence relation as defined in Lemma \ref{lm:parabolic}. 
By induction, we may assume that the coherent configuration $\wl{(X_0)}$, 
where $X_0$ is the quotient graph of $X$ modulo $e$, and 
the coherent configurations $\wl{(X_{\Delta})}$, $\Delta\in \Omega/e$, 
are all 6-separable.

Put $\wh{e}=e^{\partial}$. 
Then $\cX'$ coincides with the image of $\wh{\psi}(\cX^{\partial})$ under 
the bijection $\partial'\colon (\alpha',\alpha')\mapsto \alpha'$, $\alpha'\in\Omega'$. 
It follows that $\wh\psi$ induces an algebraic isomorphism $\varphi\colon\cX_e\to\cX'_{e'}$
where $e'=(\wh{\psi}(\wh{e}))^{\partial'}$, and such that
\begin{equation}\label{eq:f=p}
\varphi(s)=\psi(s)\mbox{~for all~} s\in S(\cX). 
\end{equation}
Therefore, $\cX$, $\cX'$ and $\varphi$ satisfy the hypothesis of Theorem \ref{theo:C}. 
Hence there exist the algebraic isomorphisms $\varphi_{\Delta,\Delta'}$ and 
$\varphi_0$ as defined 
in Theorem \ref{theo:C}(ii) and (iii). Further, by Theorem \ref{theo:iso}(ii), 
they are $6$-isomorphisms. Since $\wl{(X_0)}$ and $\wl{(X_{\Delta})}$, for all $\Delta\in \Omega/e$, are 6-separable, $\varphi_0$ and all $\varphi_{\Delta,\Delta'}$ are induced by isomorphisms. 
Thus, by Theorem \ref{theo:iso}(i), $\varphi$ is also induced by isomorphism. 
By Eq. \eqref{eq:f=p}, it follows that $\psi$ is induced by isomorphism, 
whence $\cX$ is 6-separable.

\bibliographystyle{amsplain}
\bibliography{permutations}

\end{document}